\documentclass[11pt,a4paper]{amsart}
\usepackage{amsfonts}
\usepackage{amsthm}
\usepackage{amsmath}
\usepackage{amscd}
\usepackage[latin2]{inputenc}
\usepackage{t1enc}
\usepackage[mathscr]{eucal}
\usepackage{indentfirst}
\usepackage{graphicx}
\usepackage{graphics}

\numberwithin{equation}{section}
\usepackage[margin=2.9cm]{geometry}
\usepackage{epstopdf} 
\usepackage{mathrsfs}

\theoremstyle{plain}
\newtheorem{thm}{Theorem}[section]
\newtheorem{cor}[thm]{Corollary}
\newtheorem{lem}[thm]{Lemma}

\theoremstyle{definition}

\theoremstyle{remark}

\begin{document}
	
\title[\(K\)-inner functions and \(K\)-contractions]{\(K\)-inner functions and \(K\)-contractions}

\author[J\"org Eschmeier]{J\"org Eschmeier}

\address{%
	Fachrichtung Mathematik\\
	Universit\"at des Saarlandes\\
	Postfach 15 11 50\\
	D-66041\\
	Saarbr\"ucken
	Germany}

\email{eschmei@math.uni-sb.de}

\author[Sebastian Toth]{Sebastian Toth}

\address{
	Fachrichtung Mathematik\\
	Universit\"at des Saarlandes\\
	Postfach 15 11 50\\
	D-66041\\
	Saarbr\"ucken
	Germany}
\email{toth@math.uni-sb.de}

\begin{abstract}
For a large class of unitarily invariant reproducing kernel functions $K$ on the 
unit ball $\mathbb B_d$ in $\mathbb C^d$, we characterize the $K$-inner functions on
$\mathbb B_d$ as functions admitting a suitable transfer function realization. We
associate with each $K$-contraction $T \in L(H)^d$ a canonical operator-valued
$K$-inner function and extend a uniqueness theorem of Arveson for minimal
$K$-dilations to our setting. We thus generalize results of Olofsson for $m$-hypercontractions on
the unit disc and of the first named author for $m$-hypercontractions on
the unit ball.\\
	
\vspace{0.5cm}
	
\end{abstract}

\maketitle

\section{Introduction}

Let $\mathbb B_d \subset \mathbb C^d$ be the open Euclidean unit ball and let
\begin{displaymath}
	k\colon\mathbb{D}\to\mathbb{C},\thickspace k(z)=\sum_{n=0}^{\infty}a_{n}z^{n}
\end{displaymath}
be an analytic function without zeros on the unit disc $\mathbb D$ in $\mathbb C$ such that 
\(a_{0}=1, a_{n}>0\) for all $n \in \mathbb N$ and such that
\begin{displaymath}
	0 < \inf_{n\in\mathbb{N}}\frac{a_{n}}{a_{n+1}} \leq \sup_{n\in\mathbb{N}}\frac{a_{n}}{a_{n+1}} < \infty.
\end{displaymath}
Since $k$ has no zeros, the reciprocal function $1/k \in \mathscr O(\mathbb D)$ admits a Taylor expansion
\[
(1/k)(z) = \sum_{n=0}^{\infty} c_n z^n \quad \; (z \in \mathbb D).
\]
The reproducing kernel
\begin{displaymath}
	K\colon\mathbb{B}_{d}\times\mathbb{B}_{d}\to\mathbb{C},K(z,w)=k(\langle z,w\rangle)
\end{displaymath}
defines an analytic functional Hilbert space \(H_{K}\) such that the row operator 
\(M_{z}\colon H_{K}^d\to H_{K}\) is bounded and has closed range  (\cite[Theorem A.1]{Arv2007}). Typical examples of 
functional Hilbert spaces of this type on the unit ball $\mathbb{B}_d$
are the Drury-Arveson space, the Dirichlet space, the Hardy space and the weighted Bergman spaces. 

Let $T = (T_1, \ldots ,T_d) \in L(H)^d$ be a commuting tuple of bounded linear operators on a complex Hilbert space $H$
and let $\sigma_T: L(H) \rightarrow L(H)$ be the map defined by $\sigma_T(X) = \sum_{i=1}^d T_i X T_i^*$. 
The tuple $T$ is called a $K$-contraction
if the limit
\[
\frac{1}{K}(T) = {\rm SOT-}\sum_{n=0}^{\infty}c_n\sigma_T^n(1_H) = 
 {\rm SOT-}\sum_{\alpha \in \mathbb N^d}c_{|\alpha |}\gamma_{\alpha}T^{\alpha}T^{*\alpha}
\]
exists and defines a positive operator. Here $\gamma_{\alpha} = |\alpha|!/\alpha !$ for $\alpha \in \mathbb N^d$. 

If $K(z,w) = 1/(1 -\langle z,w \rangle)$ is the Drury-Arveson kernel, then under a natural pureness condition the 
$K$-contractions coincide with the commuting row contractions of class $C_{\cdot 0}$. If $m$ is a
positive integer and $K_m(z,w) = 1/(1 -\langle z,w \rangle)^m$, then the pure $K_m$-contractions are
precisely the row-$m$-hypercontractions of class $C_{\cdot 0}$ (\cite[Theorem 3.49]{Sch18} and \cite[Lemma 2]{MV}).

An operator-valued analytic function $W: \mathbb B_d \rightarrow L(\mathscr E_*,\mathscr E)$ with Hilbert spaces
$\mathscr E$ and $\mathscr E_*$ is called 
$K$-inner if the map $\mathscr E_* \rightarrow H_K(\mathscr E)$, $x \mapsto Wx$, is a well-defined
isometry and
\[
(W\mathscr E_*) \perp M_z^{\alpha}(W\mathscr E_*) \quad \; \mbox{ for all } \alpha \in \mathbb N^d \setminus \{ 0 \}.
\]
Here $H_K(\mathscr E)$ is the $\mathscr E$-valued functional Hilbert space on $\mathbb B_d$ with reproducing kernel
$K_{\mathscr E}: \mathbb{B}_d \times \mathbb{B}_d \to L(\mathscr E), (z,w) \mapsto K(z,w)1_{\mathscr E}$. 

It was shown by Olofsson \cite{Olo06} that, for $d = 1$ and the Bergman-type kernel
\[
K_m: \mathbb D \times \mathbb D \to \mathbb C, \; K_m(z,w) = \frac{1}{(1 - z \overline{w})^m} \quad (m \in \mathbb N \setminus \{0\}),
\]
the $K_m$-inner functions $W: \mathbb D \rightarrow L(\mathscr E_*,\mathscr E)$ are precisely the
functions of the form
\[
W(z) = D + C \sum^m_{k=1} (1 - zT^*)^{-k}B,
\]
where $T \in L(H)$ is a pure $m$-hypercontraction on some Hilbert space $H$ and $B \in L(\mathscr E_*,H)$,
$C \in L(H,\mathscr E)$ and $D \in L(\mathscr E_*,\mathscr E)$ are bounded operators satisfying the 
operator equations
\begin{align*}
	&C^{*}C = (1/K_m)(T),\\
	&D^{*}C+B^{*}\Delta_{T}T^{*} = 0,\\
	&D^{*}D+B^{*}\Delta_{T}B = 1_{\mathscr E_*},
\end{align*}
where $(1/K_m)(T)$ is the $m$-th order defect operator of $T$ and
\[
\Delta_T = \sum_{k=0}^{m-1} (-1)^k \binom{m}{k+1} T^k T^{*k}.
\]

In \cite{Esc18} the result of Olofsson was extended to the unit ball by showing that a corresponding
characterization holds for functions $W: \mathbb B_d \to L(\mathscr E_*,\mathscr E)$ that are 
$K_m$-inner with respect
to the generalized Bergman kernels
\[
K_m: \mathbb B_d \times \mathbb B_d \to \mathbb C, \; K_m(z,w) = 1/(1 - \langle z,w \rangle)^m. 
\]                                                                                 

In the present note we show that the same result holds true for a large class of kernels
\[
K: \mathbb B_d \times \mathbb B_d \to \mathbb C, \; K(z,w) = \sum_{n=0}^{\infty} a_n \langle z,w \rangle^n
\]
including all complete Nevanlinna-Pick kernels such as the Drury-Arveson and the Dirichlet kernel and
all powers $K_{\nu}(z,w) = 1/(1 -\langle z,w \rangle)^{\nu}$ of the Drury-Arveson kernel with positive
real exponents. To prove that each $K$-inner function admits a transfer function realization as 
described above we extend a uniqueness result for minimal $K$-dilations due to Arveson to our
class of kernels.

\section{Wandering subspaces}\label{sectionwanderingsubspaces}

Let $T = (T_1, \ldots ,T_d) \in L(H)^d$ be a $K$-contraction, that is, a commuting tuple of bounded linear 
operators on a complex Hilbert space $H$ such that the limit
\[
\frac{1}{K}(T) = {\rm SOT-}\sum_{n=0}^{\infty}c_n\sigma_T^n(1_H) = 
 {\rm SOT-}\sum_{\alpha \in \mathbb N^d}c_{|\alpha |}\gamma_{\alpha}T^{\alpha}T^{*\alpha}
\]
exists and defines a positive operator.  A $K$-contraction $T \in L(H)^d$ is said to be pure if
\[
{\rm SOT-}\lim_{N\to\infty} 1_H - \sum_{n=0}^{N}a_n\sigma_T^n (\frac{1}{K}(T))=0.
\]

Let us define the defect operator and the defect space of a \(K\)-contraction $T$ by
\begin{displaymath}
	C = \frac{1}{K}(T)^{\frac{1}{2}}\text{ and }\mathscr{D} = \overline{\operatorname{Im}C}.
\end{displaymath}
We call an isometric linear map \(j\colon H \to H_K(\mathscr{E})\) which intertwines the tuples 
\(T^{*}\in L(H)^{d}\) and \(M_{z}^{*}\in L(H_K(\mathscr{E}))^{d}\) componentwise a $K$-dilation
of $T$. By definition a $K$-dilation \(j\colon H \to H_K(\mathscr{E})\) is minimal if the only 
reducing subspace of \(M_{z}\in L(H_K(\mathscr{E}))^{d}\) that contains the image of $j$ is 
\(H_K(\mathscr{E})\). 

Exactly as for row-\(m\)-hypercontractions of class \(C._{0}\), one can construct a canonical $K$-dilation 
for each \(K\)-contraction. 

\begin{thm}\label{thminter}
Let \(T\in L(H)^{d}\) be a pure \(K\)-contraction. Then
\begin{displaymath}
	j\colon H\to H_K(\mathscr{D}),\thickspace j(h)=\sum_{\alpha\in\mathbb{N}^d}a_{\vert\alpha\vert}\gamma_{\alpha}CT^{*\alpha}hz^{\alpha}
\end{displaymath}	
is a well defined isometry such that
\begin{displaymath}
	jT^{*}_{i}=M^{*}_{z_{i}}j\quad(i=1,\dots,d).
\end{displaymath}
\end{thm}

For a proof, see \cite[Theorem 2.15]{Sch18}. For \(h\in H\) and 
\(f=\sum_{\alpha\in\mathbb{N}^{d}}f_{\alpha}z^{\alpha}\in H_{K}(\mathscr{D})\)
\begin{displaymath}
	\left\langle h,j^{*}f\right\rangle=\sum_{\alpha\in\mathbb{N}^{d}}\left\langle CT^{*\alpha}h,f_{\alpha}\right\rangle=\sum_{\alpha\in\mathbb{N}^{d}}\left\langle h,T^{\alpha}Cf_{\alpha}\right\rangle.
\end{displaymath}
An application of the uniform boundedness principle shows that the adjoint \(j^{*}\colon H_{K}(\mathscr{D})\to H\) of the isometry \(j\) acts as
\begin{displaymath}
 j^{*}\left(\sum_{\alpha\in\mathbb{N}^{d}}f_{\alpha}z^{\alpha}\right)=\sum_{\alpha\in\mathbb{N}^{d}}T^{\alpha}Cf_{\alpha}.
\end{displaymath}

 Since \(j\) intertwines \(T^{*}\) and \(M_{z}^{*}\) componentwise, the space
\begin{displaymath}
	M = H_K(\mathscr{D})\ominus\operatorname{Im}j\subset H_K(\mathscr{D})
\end{displaymath} 
is invariant for \(M_{z}\in L(H_K(\mathscr{D}))^{d}\). 

In the following we show that the wandering subspace of $M_z$ restricted to $M$ can be described in terms of a
suitable $K$-inner function. Recall that a closed subspace $\mathscr W \subset H$ is called a
wandering subspace for a commuting tuple $S \in L(H)^d$ if
\[
\mathscr W \perp S^{\alpha} \mathscr W  \quad \quad (\alpha \in \mathbb N^d \setminus \{0\}).
\]
The space $\mathscr W$ is called a generating wandering subspace for $S$ if in addition
$H = \bigvee(S^{\alpha} \mathscr W; \alpha \in \mathbb N^d)$. For each closed $S$-invariant subspace
$L \subset H$, the space
\[
W_S(L) = L \ominus \sum_{i=1}^d S_i L
\]
is a wandering subspace for $S$, usually called the wandering subspace associated with $S$ on $L$.
If $\mathscr W$ is a generating wandering subspace for $S$, then an elementary argument shows
that necessarily $\mathscr W = W_S(H)$.

In the following we write
\begin{displaymath}
	W(M) = M\ominus\left(\sum_{i=1}^{d}M_{z_{i}}M\right)
\end{displaymath}
for the wandering subspace associated with the restriction of $M_z$ to the invariant subspace $M = {\rm Im}\,j$. 
Our main tool will be the matrix operator
\begin{align*}
	M_{z}^{*}M_{z}=(M_{z_{i}}^{*}M_{z_{j}})_{1\leq i,j\leq d}\in L(H_K(\mathscr{D})^{d}).
\end{align*}
Since the row operator \(M_{z}\colon H_{K}(\mathscr{D})^{d}\to H_{K}(\mathscr{D})\) has closed range, the operator
\begin{displaymath}
	M_{z}^{*}M_{z}\colon\operatorname{Im}M_{z}^{*}\to\operatorname{Im}M_{z}^{*}
\end{displaymath}
is invertible. We denote its inverse by \((M_{z}^{*}M_{z})^{-1}\).
In the following we consider the operators
\begin{displaymath}
	\delta\colon H_K(\mathscr{D})\to H_K(\mathscr{D}),\thickspace\delta\left(\sum_{n=0}^{\infty}\sum_{\vert\alpha\vert=n}f_{\alpha}z^{\alpha}\right)=f_{0}+\sum_{n=1}^{\infty}\frac{a_{n}}{a_{n-1}}\sum_{\vert\alpha\vert=n}f_{\alpha}z^{\alpha}
\end{displaymath}
and
\begin{displaymath}
\Delta\colon H_K(\mathscr{D})\to H_K(\mathscr{D}),\thickspace\Delta\left(\sum_{n=0}^{\infty}\sum_{\vert\alpha\vert=n}f_{\alpha}z^{\alpha}\right)=\sum_{n=0}^{\infty}\frac{a_{n+1}}{a_{n}}\sum_{\vert\alpha\vert=n}f_{\alpha}z^{\alpha}.
\end{displaymath}

By definition \(\delta\) and \(\Delta\) are diagonal operators with respect to the orthogonal decomposition 
\(H_{K}( \mathscr{ D})=\oplus^\infty_{n=0}H_{n}( \mathscr{D})\) of \(H_{K}(\mathscr {D})\) into the spaces 
\(H_{n}( \mathscr{D})\) of all \(\mathscr{D}\)-valued homogenous polynomials of degree \(n\). Our hypotheses 
on the sequence $(a_n/a_{n+1})$ imply that \(\delta\) and \(\Delta\) are invertible 
positive operators on \(H_{K}(\mathscr{D})\). An elementary calculation shows that
\begin{displaymath}
\delta M_{z_{i}}=M_{z_{i}}\Delta
\end{displaymath}
for \(i=1,\ldots,d\). 
\begin{lem}\label{lemsdelta}
For \(f\in H_{K}(\mathscr{D})\), we have
\begin{displaymath}
	(M_{z}^{*}M_{z})^{-1}(M_{z}^{*}f)=M_{z}^{*}\delta f=\left(\oplus \Delta\right)M_{z}^{*} f.
\end{displaymath}
In particular the row operator
\begin{displaymath}	
	\delta M_{z}\colon H_{K}(\mathscr{D})^{d}\to H_{K}(\mathscr{D})
\end{displaymath}
defines the trivial extension of the operator
\begin{displaymath}
	M_{z}\left(M_{z}^{*}M_{z}\right)^{-1}\colon\operatorname{Im}M_{z}^{*}\to H_{K}(\mathscr{D}).
\end{displaymath}
\end{lem}
\begin{proof}
	Since the column operator \(M_{z}^{*}\) annihilates the constant functions, to prove the first identity, 
	we may suppose  that \(f(0)=0\). With respect to the orthogonal decomposition 
	\(H_{K}(\mathscr{D})=\bigoplus_{n=0}^{\infty}H_{n}(\mathscr{D})\) the operator \(M_{z}M_{z}^{*}\) acts as (Lemma 4.3 in \cite{Guo04})
	\begin{displaymath}
		M_{z}M_{z}^{*}\left(\sum_{n=0}^{\infty}f_{n}\right)=\sum_{n=1}^{\infty}\left(\frac{a_{n-1}}{a_{n}}\right)f_{n}.
	\end{displaymath}
	Hence \(M_{z}M_{z}^{*}\delta f=f\) and
	\begin{displaymath}
		\left(M_{z}^{*}M_{z}\right)^{-1}M_{z}^{*}f = 
		\left(M_{z}^{*}M_{z}\right)^{-1}\left(M_{z}^{*}M_{z}\right)M_{z}^{*}\delta f=M_{z}^{*}\delta f = (\oplus \Delta)M_z^*f.
	\end{displaymath}
	Since any two diagonal operators commute, it follows in particular that 
	\(M_{z}\left(M_{z}^{*}M_{z}\right)^{-1}M_{z}^{*}=\delta\left(M_{z}M_{z}^{*}\right)\). Thus also the second assertion follows.
\end{proof}

The preceding proof shows in particular that the orthogonal projection of \(H_{K}(\mathscr{D})\) onto \(\operatorname{Im}M_{z}\) acts as
\begin{displaymath}
	P_{\operatorname{Im}M_{z}}=	M_{z}(M_{z}^{*}M_{z})^{-1}M_{z}^{*}=\delta (M_{z}M_{z}^{*}) = P_{H_K(\mathscr D)\ominus \mathscr D},
\end{displaymath}
where \(\mathscr{D}\subset H_{K}(\mathscr{D})\) is regarded as the closed subspace consisting of all constant functions.
As in the single-variable case we call the operator defined by 
\(M_{z}^{\prime}=\delta M_{z}\in L(H_{K}(\mathscr{D})^{d},H_{K}(\mathscr {D}))\) the Cauchy dual of the 
multiplication tuple \(M_{z}\).

We use the operator $\Delta_{T} \in L(H)$ defined by
\begin{displaymath}
	\Delta_{T}=j^{*}\Delta j
\end{displaymath}
to give a first desciption of the wandering subspace \(W(M)\) of $M_z$ restricted to the invariant 
subspace \(M=(\operatorname{Im}j)^\perp\).

\begin{thm}\label{thmwand1}
	A function \(f\in H_{K}(\mathscr{D})\) is an element of the wandering subspace \(W(M)\) of \(M=(\operatorname{Im}j)^{\perp}\in{\operatorname{ Lat}}(M_{z},H_{K}(\mathscr{D}))\) if and only if
	\begin{displaymath}
		f=f_0+M_{z}^{\prime}(jx_{i})_{i=1}^{d}
	\end{displaymath} 
	for some vectors \(f_0\in\mathscr{D}\), \(x_1,\ldots,x_d\in H\) with \((jx_{i})_{i=1}^{d}\in M_{z}^{*}H_{K}(\mathscr {D})\) and
	\begin{displaymath}
		Cf_0+T(\Delta_{T}x_{i})_{i=1}^{d}=0.
	\end{displaymath}
	In this case \((jx_{i})_{i=1}^{d}=M_{z}^{*}f\).
	\end{thm}
\begin{proof}
	Note that a function \(f\in H_{K}(\mathscr{D})\) belongs to the wandering subspace 
	\(W(M)=M\ominus\sum_{i=1}^{d}z_{i}M\) of \(M_{z}\) on 
	\(M=\operatorname{Ker}j^{*}\in\operatorname{Lat}(M_{z},H_{K}(\mathscr{D}))\) if and 
	only if \(j^{*}f=0\) and \((1_{H_{K}(\mathscr{D})}-jj^{*})M_{z_{i}}^{*}f=0\) for \(i=1,\ldots,d\). 
	Using the remark following Lemma \ref{lemsdelta}, we obtain, for \((x_{i})_{i=1}^{d}\in H^{d}\) 
	and \(f\in H_{K}(\mathscr{D})\) with \((jx_{i})_{i=1}^{d}=M_{z}^{*}f\),
	\begin{align*}
		j^{*}f&=j^{*}(f(0)+\delta M_{z}M_{z}^{*}f)\\
		&=Cf(0)+j^{*}M_{z}(\Delta jx_{i})_{i=1}^{d}\\
		&=Cf(0)+T(j^{*}\Delta jx_{i})_{i=1}^{d}\\
		&=Cf(0)+T(\Delta_{T}x_{i})_{i=1}^{d}.
	\end{align*} 
	Thus if \(f\in W(M)\), then \((x_{i})_{i=1}^{d}=(j^{*}M_{z_{i}}^{*}f)_{i=1}^{d}\) defines a tuple in \(H^{d}\) with \((jx_{i})_{i=1}^{d}=M_{z}^{*}f\) such that \(Cf(0)+T(\Delta_{T}x_{i})_{i=1}^{d}=j^{*}f=0\) and 
	\begin{align*}
		f=f(0)+(f-f(0))=f(0)+M_{z}(M_{z}^{*}M_{z})^{-1}M_{z}^{*}f
		=f(0)+M_{z}^{\prime}(jx_{i})_{i=1}^{d}.
	\end{align*}
	Conversely, if \(f=f_{0}+M_{z}^{\prime}(jx_{i})_{i=1}^{d}\) with \(f_{0}\in\mathscr{D}\), \(x_{1},\ldots,x_{d}\) as in Theorem \ref{thmwand1}, then using Lemma \ref{lemsdelta} we find that 
	\begin{displaymath}
		M_{z}^{*}f=M_{z}^{*}M_{z}(M_{z}^{*}M_{z})^{-1}(jx_{i})_{i=1}^{d}=(jx_{i})_{i=1}^{d}.
	\end{displaymath}
	Since \(j\) is an isometry, it follows that \(jj^{*}M_{z_{i}}^{*}f=jx_{i}=M_{z_{i}}^{*}f\) for \(i=1,\ldots,d\). Since \(j^{*}f=Cf(0)+T(\Delta_{T}x_{i})_{i=1}^{d}=0\), we have shown that \(f\in W(M)\).
\end{proof}

\begin{lem}\label{normwand}
Let \(T \in L(H)^d\) be a pure \(K\)-contraction and let
\begin{displaymath}
f=f_0+M_{z}^{\prime}(jx_{i})_{i=1}^{d}
\end{displaymath}
be a representation of a function \(f\in W(M)\) as in Theorem \ref{thmwand1}.
Then we have
\begin{displaymath}
\Vert f\Vert^{2}=\Vert f_0\Vert^2+\sum_{i=1}^{d}\langle\Delta_{T}x_{i},x_{i}\rangle.
\end{displaymath}
\end{lem}

\begin{proof}
	Since by Lemma \ref{lemsdelta}
	\begin{displaymath}
		\operatorname{Im}{M_{z}^{\prime}}=M_{z}(M_{z}^{*}M_{z})^{-1}M_{z}^{*}H_{K}(\mathscr{D})=\operatorname{Im}{M_{z}}=H_K(\mathscr{D})\ominus\mathscr{D},
	\end{displaymath}
	it follows that
	\begin{align*}
		\Vert f\Vert^{2}-\Vert f_{0}\Vert^{2}
		&=\Vert M_{z}^{\prime}(jx_{i})_{i=1}^{d}\Vert^{2}\\
		&=\langle(M_{z}^{*}M_{z})^{-1}M_{z}^{*}f,(jx_{i})_{i=1}^{d}\rangle\\
		&=\langle(\oplus j^{*})M_{z}^{*}\delta f,(x_{i})_{i=1}^{d}\rangle\\
		&=\langle(j^{*}\Delta jx_{i})_{i=1}^{d},(x_{i})_{i=1}^{d}\rangle.
	\end{align*}
	Since by definition \(\Delta_{T}=j^{*}\Delta j\), the assertion follows.
\end{proof}

Let \(T\in L(H)^{d}\) be a pure \(K\)-contraction.
Then \(\Delta_{T}=j^{*}\Delta j\) is a positive operator with
\begin{displaymath}
	\langle\Delta_{T}x,x\rangle=\Vert\Delta^{\frac{1}{2}}jx\Vert^{2}\geq\Vert\Delta^{-\frac{1}{2}}\Vert^{-2}\Vert jx\Vert^{2}=\Vert\Delta^{-1}\Vert^{-1}\Vert x\Vert^{2}
\end{displaymath}
for all \(x\in H\). Hence \(\Delta_{T}\in L(H)\) is invertible and
\begin{displaymath}
(x,y)=\langle\Delta_{T}x,y\rangle
\end{displaymath}
defines a scalar product on \(H\) such that the induced norm \(\Vert\cdot \Vert_{T}\) is equivalent to the original norm with
\begin{displaymath}
	\Vert\Delta^{\frac{1}{2}}\Vert\Vert x\Vert\geq\Vert x\Vert_{T}\geq\Vert\Delta^{-\frac{1}{2}}\Vert^{-1}\Vert x\Vert
\end{displaymath}
for \(x\in H\). We write \(\tilde{H}\) for \(H\) equipped with the norm \(\Vert\cdot\Vert_{T}\). Then 
\begin{displaymath}
	I_{T}\colon H\to \tilde{H},\thickspace x\mapsto x
\end{displaymath}
is an invertible bounded operator such that
\begin{displaymath}
\langle	I_{T}^{*}x,y\rangle=\langle\Delta_{T}x,y\rangle\quad(x\in \tilde{H},y\in H).
\end{displaymath}
Hence \(I_{T}^{*}x=\Delta_{T}x\) for \(x\in\tilde{H}\). Let \(\tilde{T}=(\tilde{T}_{1},\ldots,\tilde{T}_{d})\colon\tilde{H}^{d}\to H\) be the row operator with components
\(\tilde{T}_{i}=T_{i}\circ I_{T}^{*}\in L(\tilde{H},H)\). Then 
\begin{align*}
\tilde{T}\tilde{T}^{*}
&=\sum_{i=1}^{d}T_{i}(I_{T}^{*}I_{T})T_{i}^{*}
=\sigma_{T}(\Delta_{T})
=\sigma_{T}(j^{*}\Delta j)
=j^{*}M_{z}(\oplus\Delta)M_{z}^{*}j\\
&=j^{*}(\delta M_{z}M_{z}^{*})j
=j^{*}P_{H_K(\mathscr D)\ominus \mathscr D}j
\end{align*}
and hence \(\tilde{T}\) is a contraction. As in \cite{Olo06} we use its defect operators
\begin{align*}
D_{\tilde{T}}&=(1_{\tilde{H}^d}-\tilde{T}^{*}\tilde{T})^{1/2}\in L(\tilde{H}^d),\\
D_{\tilde{T}^{*}}&=(1_{H}-\tilde{T}\tilde{T}^{*})^{1/2}=(j^{*}P_{\mathscr{D}}j)^{1/2}=C\in L(H).\\
\end{align*}
Here the identity \((j^{*}P_{\mathscr{D}}j)^{1/2}=C\) follows from the definition 
of \(j\) and the representation of \(j^{*}\) explained in 
the section following Theorem \ref{thminter}. We write 
\(\mathscr{D}_{\tilde{T}}=\overline{D_{\tilde{T}}\tilde{H}^d}\subset \tilde{H}^{d}\) and 
\(\mathscr{D}_{\tilde{T}^{*}}=\overline{D_{\tilde{T}^{*}}H}=\mathscr{D}\) for the defect spaces of \(\tilde{T}\). 
As in the classical single-variable theory of contractions it follows that 
\(\tilde{T}D_{\tilde{T}} = D_{\tilde{T}^{*}}\tilde{T}\) and that
\begin{displaymath}
U=
\left(
\begin{array}{c|c}
\tilde{T}&D_{\tilde{T}^{*}}\\
\hline
D_{\tilde{T}}&-\tilde{T}^{*}
\end{array}
\right)\colon\tilde{H}^{d}\oplus\mathscr{D}_{\tilde{T}^{*}}\to H\oplus \mathscr{D}_{\tilde{T}}
\end{displaymath}
is a well-defined unitary operator. In the following we define an analytically parametrized 
family \(W_{T}(z)\in L(\tilde{\mathscr{D}}, \mathscr{D})\) \((z\in \mathbb B)\) of operators on the closed 
subspace
\begin{displaymath}
	\tilde{\mathscr{D}} = \lbrace y\in\mathscr{D}_{\tilde{T}};
	\thickspace(\oplus jI_{T}^{-1})D_{\tilde{T}}y\in M_{z}^{*}H_{K}(\mathscr{D})\rbrace\subset\mathscr{D}_{\tilde{T}}
\end{displaymath}
such that
\begin{displaymath}
	W(M)=\lbrace W_{T}x;\thickspace x\in\tilde{\mathscr{D}}\rbrace,
\end{displaymath}
where \(W_{T}x\colon\mathbb{B}_{d}\to\mathscr{D}\)  acts as \((W_{T}x)(z)=W_{T}(z)x\). 
We equip \(\tilde{\mathscr{D}}\) with the norm \(\Vert y\Vert=\Vert y\Vert_{\tilde{H}^{d}}\) 
that it inherits as a closed subspace \(\tilde{\mathscr{D}}\subset\tilde{H}^d\). 

\begin{lem}\label{lemwand2}
Let \(T \in L(H)^{d}\) be a pure \(K\)-contraction. Then a function \(f\in H_{K}(\mathscr{D})\) belongs to the wandering subspace \(W(M)\) of
\begin{displaymath}
	M=(\operatorname{Im}j)^{\perp}\in\operatorname{Lat}(M_{z},H_{K}(\mathscr{ D}))
\end{displaymath}
if and only if there is a vector \(y \in \tilde{\mathscr{D}}\) with
\begin{displaymath}
	f = -\tilde{T}y+M_{z}^{\prime}(\oplus jI_{T}^{-1})D_{\tilde{T}}y.
\end{displaymath}
In this case \(\Vert f\Vert^2=\Vert y\Vert^2_{\tilde{H}^{d}}\).
\end{lem}
\begin{proof}
By Theorem \ref{thmwand1} a function \(f\in H_{K}(\mathscr{D})\) belongs to \(W(M)\) if and only if  it is of the form
\begin{displaymath}
	f=f_{0}+M_{z}^{\prime}(jx_{i})_{i=1}^{d}
\end{displaymath}
with \(f_0\in\mathscr{D}\) and \(x_{1},\ldots,x_{d}\in H\) such that \((jx_{i})_{i=1}^{d}\in M_{z}^{*}H_{K}(\mathscr{D})\) and
\begin{displaymath}
	\tilde{T}(I_{T}x_{i})_{i=1}^{d}+D_{\tilde{T}^{*}}f_{0}=0.
\end{displaymath}
Then \(y=D_{\tilde{T}}(I_{T}x_{i})_{i=1}^{d}-\tilde{T}^{*}f_{0}\in\mathscr{D}_{\tilde{T}}\) is a vector with
\begin{displaymath}
	U\left(
	\begin{array}{cc}
		(I_{T}x_{i})\\
		f_{0}
	\end{array}	\right)=\left(
	\begin{array}{cc}
		0\\
		y
	\end{array}	\right),
\end{displaymath}
or equivalently, with
\begin{displaymath}
\left(
\begin{array}{cc}
(I_{T}x_{i})\\
f_{0}
\end{array}	\right)=U^{*}\left(
\begin{array}{cc}
0\\
y
\end{array}	\right)=\left(
\begin{array}{cc}
D_{\tilde{T}}y\\
-\tilde{T}y
\end{array}	\right).
\end{displaymath}
But then \(y\in\tilde{\mathscr{D}}\) and \(f=-\tilde{T}y+M_{z}^{\prime}(\oplus jI_{T}^{-1})D_{\tilde{T}}y\). 
Conversely, if \(f\) is of this form, then using the definitions of \(\tilde{T}\), \(\tilde{\mathscr{D}}\)  
and the intertwining relation \(\tilde{T}D_{\tilde{T}}=D_{\tilde{T}^{*}}\tilde{T}\) one can easily show that 
the vectors defined by
\begin{displaymath}
	f_{0}=-\tilde{T}y\in\mathscr{D}\text{ and }(x_{i})_{i=1}^{d}=(\oplus I_{T}^{-1})D_{\tilde{T}}y\in H^{d}
\end{displaymath}
yield a representation \(f=f_{0}+M_{z}^{\prime}(jx_{i})_{i=1}^{d}\) as in Theorem \ref{thmwand1}. 
By Lemma \ref{normwand} and the definition of the scalar product on $\tilde{H}$ we find that
\begin{align*}
	\Vert f\Vert^{2}
	&=\Vert f_{0}\Vert^{2}+\sum_{i=1}^{d}\langle \Delta_{T}x_{i},x_{i}\rangle
	=\Vert\tilde{T}y\Vert^{2}+\sum_{i=1}^{d}\Vert I_{T}x_{i}\Vert_{\tilde{H}}^{2}\\
	&=\Vert\tilde{T}y\Vert^{2}+\Vert D_{\tilde{T}}y\Vert_{\tilde{H}^{d}}^{2}
	=\Vert y\Vert_{\tilde{H}^d}^{2}.
\end{align*}
\end{proof}

Recall that the reproducing kernel $K: \mathbb B_d \times \mathbb B_d \rightarrow \mathbb C$ is defined
by $K(z,w) = k(\langle z,w \rangle)$, where
\[
k\colon\mathbb{D}\to\mathbb{C},\thickspace k(z)  =\sum_{n=0}^{\infty}a_{n}z^{n}
\] 
is an analytic function with \(a_0=1,\) $a_n > 0$ for all $n$ such that 
\[
0 < \inf_n \frac{a_n}{a_{n+1}} \leq \sup_n \frac{a_n}{a_{n+1}} < \infty.
\]
Let us suppose in addition that the limit
\begin{displaymath}
	r = \lim_{n\to\infty}\frac{a_{n}}{a_{n+1}}
\end{displaymath}
exists. Then $r \in [1,\infty)$ is the radius of convergence of the power series defining $k$ and
by Theorem 4.5 in \cite{Guo04} the Taylor spectrum of \(M_{z}\in L(H_{K}(\mathscr{D}))^{d}\) is given by
\begin{displaymath}
	\sigma(M_{z})=\lbrace z\in\mathbb{C}^{d}; \Vert z\Vert\leq\sqrt{r}\rbrace.
\end{displaymath}
If \(T\in L(H)^{d}\) is a pure \(K\)-contraction, then \(T^{*}\) is unitarily equivalent to a restriction of \(M_{z}^{*}\) and hence
\begin{displaymath}
\sigma(T^{*})\subset\lbrace z\in\mathbb{C}^{d}; \Vert z\Vert\leq\sqrt{r}\rbrace.
\end{displaymath}
The function \(F\colon D_{r}(0)\to\mathbb{C}, F(z)=\sum_{n=0}^{\infty}a_{n+1}z^{n}\), is analytic on the open disc
$D_r(0)$ with radius $r$ and center $0$ and satisfies
\begin{displaymath}
	F(z)=\frac{k(z)-1}{z}\quad(z\in D_{r}(0)\setminus\lbrace 0\rbrace).
\end{displaymath}
For \(z\in\mathbb{B}_{d}\), let us denote by \(Z \colon H^{d}\to H,\thickspace(h_{i})_{i=1}^{d}\mapsto\sum_{i=1}^{d}z_{i} h_{i}\),  
the row operator induced by \(z\).
As a particular case of a much more general analytic spectral mapping theorem for the Taylor spectrum 
(\cite[Theorem 2.5.10]{Esc16}) we find that
\begin{displaymath}
	\sigma(ZT^{*})=\lbrace\sum_{i=1}^{d}z_{i}w_{i};\thickspace w\in\sigma(T^{*})\rbrace\subset D_{r}(0)
\end{displaymath}
for \(z\in\mathbb{B}_{d}\). Thus we can define an operator-valued function \(F_{T}\colon\mathbb{B}_{d}\to L(H)\), 
\begin{displaymath}
	F_{T}(z) = F(ZT^{*})=\sum_{n=0}^{\infty}a_{n+1}\left(\sum_{\vert\alpha\vert=n}\gamma_{\alpha}T^{*\alpha}z^{\alpha}\right).
\end{displaymath}

\begin{lem}\label{lembigf}
	For \((x_{i})_{i=1}^{d}\in H^{d}\) and $z \in \mathbb B_d$,
	\begin{displaymath}
		CF(ZT^{*})Z(x_{i})_{i=1}^{d}=(\delta M_{z}(jx_{i})_{i=1}^{d})(z).
	\end{displaymath}
\end{lem}

\begin{proof}
	For \((x_{i})_{i=1}^{d}\in H^{d}\),
	\begin{align*}
		\delta M_{z}(jx_{i})_{i=1}^{d}
		&=\sum_{i=1}^{d}\delta M_{z_{i}}\sum_{n=0}^{\infty}a_{n}\left(\sum_{\vert\alpha\vert=n}\gamma_{\alpha}CT^{*\alpha}x_{i}z^{\alpha}\right)\\
		&=\sum_{i=1}^{d} \sum_{n=0}^{\infty}a_{n}\delta \left(\sum_{\vert\alpha\vert=n}\gamma_{\alpha}CT^{*\alpha}x_{i}z^{\alpha+e_{i}}\right)\\
		&=\sum_{i=1}^{d} \sum_{n=0}^{\infty} a_{n+1}\sum_{\vert\alpha\vert=n}\gamma_{\alpha}CT^{*\alpha}x_{i}z^{\alpha+e_{i}},
	\end{align*}
	where the series converge in \(H_{K}(\mathscr{D})\). Since the point evaluations are continuous on \(H_{K}(\mathscr{D})\), we obtain
	\begin{align*}
		\left(\delta M_{z}(jx_{i})_{i=1}^{d}\right)(z)
		&=\sum_{n=0}^{\infty} a_{n+1}\sum_{\vert\alpha\vert=n}\gamma_{\alpha}CT^{*\alpha}\left(\sum_{i=1}^{d}z_{i}x_{i}\right)z^{\alpha}\\
		&=CF(ZT^{*})Z(x_{i})_{i=1}^{d}
	\end{align*}
	for all \(z\in\mathbb{B}_{d}\).
\end{proof}

By Lemma \ref{lembigf} the map \(W_{T}\colon\mathbb{B}_{d}\to L(\tilde{\mathscr{D}},\mathscr{D})\),
\begin{align*}
	W_{T}(z)(x)&=-T(\oplus\Delta_{T}I_{T}^{-1})x+CF(ZT^{*})Z(\oplus I_{T}^{-1})  D_{\tilde{T}}x\\
	&=-\tilde{T}x+CF(ZT^{*})Z(\oplus I_{T}^{-1})  D_{\tilde{T}}x
\end{align*}
defines an analytic operator-valued function. 

\begin{thm}\label{thmwandinner}
Let \(T\in L(H)^{d}\) be a pure \(K\)-contraction. Then
\begin{displaymath}
W(M)=\lbrace W_{T}x;\thickspace x\in\tilde{\mathscr{D}}\rbrace
\end{displaymath}
and \(\Vert W_{T}x\Vert=\Vert x\Vert\) for \(x\in\tilde{\mathscr{D}}\).
\end{thm}
\begin{proof}
For \(x\in\tilde{\mathscr{D}}\), Lemma \ref{lembigf} implies that
\begin{align*}
	W_{T}x&=-\tilde{T}x+\delta M_{z}(\oplus jI_{T}^{-1})D_{\tilde{T}}x\\
	&=-\tilde{T}x+M_{z}^{\prime}(\oplus jI_{T}^{-1})D_{\tilde{T}}x.
\end{align*}
Thus the assertion follows from Lemma \ref{lemwand2}.
\end{proof}

Since \(W(M)\) is a wandering subspace for \(M_{z}\), the map
\(W_{T}\colon\mathbb{B}_{d}\to L(\tilde{\mathscr{D}},\mathscr{D})\) is an operator-valued 
analytic function such that $\tilde{\mathscr D} \rightarrow H_{K}(\mathscr{D}),$ $x \mapsto W_T x,$ 
is an isometry and
\begin{displaymath}
	W_{T}(\mathscr{\tilde{D}})\perp M_{z}^{\alpha}\left(W_{T}(\tilde{\mathscr{D}})\right)\text{ for all }\alpha\in\mathbb{N}^{d}\setminus\lbrace 0\rbrace.
\end{displaymath}
Thus $W_T: \mathbb B_d \to L(\tilde{\mathscr D},\mathscr D)$ is a $K$-inner function with
$W_T(\tilde{\mathscr D}) = W(M)$. In the case that \(M_{z}\in L(H_K)^{d}\) is a row contraction one can show that each 
\(K\)-inner function \(W\colon\mathbb{B}_{d}\to L(\tilde{\mathscr{ E}},\mathscr{E})\) defines a contractive multiplier
\begin{displaymath}
	M_{W}\colon H_{d}^{2}(\mathscr{E})\to H_{K},\thickspace f\to Wf
\end{displaymath}
from the \(\mathscr{E}\)-valued Drury-Arveson space \(H_{d}^{2}(\mathscr{E})\) to  \(H_{K}(\tilde{\mathscr{E}})\) 
(\cite[Theorem 6.2]{Bha17}).

\section{\(K\)-inner functions}\label{sectionkinnerfunctions}

In the previous section we saw that the \(K\)-inner function 
\(W_{T}\colon\mathbb{B}_{d}\to L(\tilde{\mathscr{D}},\mathscr{D})\) associated with a pure 
\(K\)-contraction \(T \in L(H)^{d}\) has the form
\begin{displaymath}
	W_{T}(z) = D + CF(ZT^{*})ZB,
\end{displaymath} 
where \(C=\left(\frac{1}{K}(T)\right)^{\frac{1}{2}}\in L(H,\mathscr{D})\), \(D=-\tilde{T}\in L(\tilde{\mathscr{D}},\mathscr{D})\) 
and \(B=(\oplus I_{T}^{-1})D_{\tilde{T}}\in L(\tilde{\mathscr{D}},H^{d})\). An elementary calculation using the definitions and the intertwining relation \(\tilde{T}D_{\tilde{T}}=D_{\tilde{T}^{*}}\tilde{T}\) shows that the operators \(T\), \(B\), \(C\), \(D\) 
satisfy the conditions
\begin{align*}
	&\text{(K1) }C^{*}C=\frac{1}{K}(T),\\
	&\text{(K2) }D^{*}C+B^{*}(\oplus\Delta_{T})T^{*}=0,\\
	&\text{(K3) }D^{*}D+B^{*}(\oplus\Delta_{T})B=1_{\tilde{\mathscr{D}}},\\
	&\text{(K4) }\operatorname{Im}((\oplus j)B)\subset M_{z}^{*}H_{K}(\mathscr{D}).
\end{align*}
If \(\mathscr{E}\) is a Hilbert space and \(C\in L(H,\mathscr{E})\) is any operator with \(C^{*}C=\frac{1}{K}(T)\), then exactly as in the proof of Proposition 2.6 from \cite{Sch18} it follows that 
\begin{displaymath}
	j_{C}\colon H\to H_{K}(\mathscr{E}),\thickspace j_{C}(x) = 
	\sum_{\alpha\in\mathbb{N}^{d}}a_{\vert\alpha\vert}\gamma_{\alpha}(CT^{*\alpha}x)z^{\alpha}
\end{displaymath}
is a well defined isometry that intertwines the tuples \(T^{*}\in L(H)^{d}\) and \(M_{z}^{*}\in L(H_{K}(\mathscr{E}))\) componentwise.
As in the section following Theorem \ref{thminter} one can show that
\[
j_C^* f = \sum_{\alpha \in \mathbb N^d}T^{\alpha}C^*f_{\alpha}
\]
for $f = \sum_{\alpha \in \mathbb N^d} f_{\alpha} z^{\alpha} \in H_K(\mathscr E)$. Hence we find that
\begin{align*}
		j_C^* \Delta j_C x
		&=j_C^* \Delta \sum_{\alpha\in\mathbb{N}^{d}}a_{\vert\alpha\vert}\gamma_{\alpha}(CT^{*\alpha}x)z^{\alpha}\\
		&=j_C^* \sum_{\alpha\in\mathbb{N}^{d}}a_{|\alpha | + 1}\gamma_{\alpha}(CT^{*\alpha}x)z^{\alpha}\\
		&=\sum_{\alpha\in\mathbb{N}^{d}}a_{|\alpha | + 1}\gamma_{\alpha} (T^{\alpha} C^* C T^{*\alpha}x)\\
		&=\sum_{\alpha\in\mathbb{N}^{d}}a_{|\alpha | + 1}\gamma_{\alpha} (T^{\alpha} \frac{1}{K}(T)T^{*\alpha}x)
	\end{align*}
for all $x \in H$. By performing the same chain of calculations with $j_C$ replaced by the canonical $K$-dilation
$j$ of $T$ from Theorem \ref{thminter} we obtain that
\[
j_C^* \Delta j_C = j^* \Delta j = \Delta_T.
\]
\\
Our next aim is to show that any matrix operator
\begin{displaymath}
	\left(
	\begin{array}{c|c}
		T^*&B\\
		\hline
		C&D
	\end{array}
	\right)\colon H\oplus\mathscr{E}_{*}\to H^{d}\oplus\mathscr{E},
\end{displaymath}
where \(T\) is a pure \(K\)-contraction and \(T\), \(B\), \(C\), \(D\) satisfy the conditions (K1)-(K3) with \((\tilde{\mathscr{D}},\mathscr{D})\)
replaced by \((\mathscr{E}_{*},\mathscr{E})\) and 
\begin{displaymath}
	\text{(K4) }\operatorname{Im}((\oplus j_{C})B)\subset M_{z}^{*}H_{K}(\mathscr{E})
\end{displaymath}
gives rise to a \(K\)-inner function \(W\colon\mathbb{B}_{d}\to L(\mathscr{E}_{*},\mathscr{E})\) defined as
\begin{displaymath}
	W(z) = D + CF(ZT^{*})ZB
\end{displaymath} 
and that, conversely, under a natural condition on the kernel \(K\) each \(K\)-inner function is of this form.

\begin{thm}\label{sufficient}
Let \(W\colon\mathbb{B}_{d}\to L(\mathscr{E}_{*},\mathscr{E})\) be an operator-valued function between Hilbert spaces 
\(\mathscr{E}_{*}\) and \(\mathscr{E}\) such that
\begin{displaymath}
	W(z) = D + CF(ZT^{*})ZB  \quad (z \in \mathbb B_d),
\end{displaymath}
where \(T\in L(H)^{d}\) is a pure \(K\)-contraction and the matrix operator
\begin{displaymath}
\left(
\begin{array}{c|c}
 T^{*}&B\\
\hline
C&D
\end{array}
\right)\colon H\oplus\mathscr{E}_{*}\to H^{d}\oplus\mathscr{E}
\end{displaymath}
satisfies the condition (K1)-(K4). Then \(W\) is a \(K\)-inner function.
\end{thm}
\begin{proof}
The space \(M=H_{K}(\mathscr{E})\ominus\operatorname{Im}j_{C}\subset H_{K}(\mathscr{E})\) is a closed \(M_{z}\)-invariant subspace. Let \(x\in\mathscr{E}_{*}\) be a fixed vector. By condition (K4) there is a function \(f\in H_{K}(\mathscr{E})\) with \((\oplus j_{C})Bx=M_{z}^{*}f\). Exactly as in the proof of Lemma \ref{lembigf} it follows that
\begin{displaymath}
	CF(ZT^{*})ZBx  =\delta M_{z}(\oplus j_{C})Bx(z) = \delta  M_{z}M_{z}^{*}f(z)
\end{displaymath}	
for all $z \in \mathbb B_d$.
Since \(\delta (M_{z}M_{z}^{*})=P_{\operatorname{Im}M_{z}}\) is an orthogonal projection and since 
\(\delta M_{z}=M_{z}(\oplus\Delta)\), we find that
\begin{align*}
	\Vert Wx\Vert_{H_{K}(\mathscr{E})}^{2}-\Vert Dx\Vert^{2}
	&=\langle \delta M_{z}M_{z}^{*}f,f\rangle_{H_{K}(\mathscr{E})}\\
	&=\langle \oplus(j_{C}^{*}\Delta j_{C})Bx,Bx\rangle_{H^{d}}\\
	&=\langle (\oplus\Delta_{T})Bx,Bx\rangle_{H^{d}}\\
	&=\langle (1_{\mathscr{E}_{*}}-D^{*}D)x,x\rangle\\
	&=\Vert x\Vert^{2}-\Vert Dx\Vert^{2}.
\end{align*}
Hence the map \(\mathscr{E}_{*}\to H_{K}(\mathscr{E}),\thickspace x\mapsto Wx\), is a well-defined isometry. Using the second part of Lemma \ref{lemsdelta} we obtain
\begin{displaymath}
	M_{z}^{*}(Wx)=M_{z}^{*}\delta M_{z}M_{z}^{*}f=M_{z}^{*}f=(\oplus j_{C})Bx
\end{displaymath} 	
and hence that \(P_{M}M_{z_{i}}^{*}(Wx)=(1_{H_{K}(\mathscr{E})}-j_{C}j_{C}^{*})M_{z_{i}}^{*}(Wx)=0\) for \(i=1,\cdots,d\). 
To see that \(W\mathscr{E}_{*}\subset M\) note that with \(x\) and \(f\) as above
\begin{align*}
	j_{C}^{*}(Wx)&=C^{*}Dx+j_{C}^{*}(\delta M_{z}M_{z}^{*}f)\\
	&=C^{*}Dx+j_{C}^{*}(M_{z}(\oplus\Delta)M_{z}^{*}f)\\
	&=C^{*}Dx+T(\oplus j_{C}^{*}\Delta j_{C})Bx\\
	&=C^{*}Dx+T(\oplus\Delta_{T})Bx\\
	&=0.
\end{align*}
Thus we have shown that \(W\mathscr{E}_{*}\subset M\ominus\sum_{i=1}^{d}z_{i}M\) which implies that
\begin{displaymath}
	W\mathscr{E}_{*}\perp z^{\alpha}(W\mathscr{E}_{*})
\end{displaymath}
for all \(\alpha \in\mathbb{N}^{d}\setminus\lbrace 0 \rbrace\).
\end{proof}

To prove that conversely each $K$-inner function $W: \mathbb B_d \rightarrow L(\mathscr E_*,\mathscr E)$
has the form described in Theorem \ref{sufficient} we make the additional assumption that the multiplication
tuple $M_z \in L(H_K)^d$ is a $K$-contraction. This hypothesis is satisfied, for instance, if $H_K$ is
a complete Nevanlinna-Pick space such as the Drury-Arveson space or the Dirichlet space or if $K$ is a power
\[
K_{\nu}: \mathbb B_d \times \mathbb B_d \rightarrow, K_{\nu}(z,w) = \frac{1}{(1 - \langle z,w \rangle)^{\nu}} \quad (\nu \in (0,\infty))
\]
of the Drury-Arveson kernel (see the discussion following Theorem \ref{thmvonneumannmz}).
In the proof we shall use a uniqueness result for minimal $K$-dilations whose proof we postpone to Section \ref{sectionminimalkdilations}.

\begin{thm}
	 Let \(M_{z}\in L(H_K)^{d}\) be a \(K\)-contraction. 
	If  \(W\colon\mathbb{B}_{d}\to L(\mathscr{E}_{*},\mathscr{E})\) is a \(K\)-inner function, 
	then there exist a pure \(K\)-contraction \(T \in L(H)^{d}\) and a matrix operator
	\begin{displaymath}
		\left(\begin{array}{c|c}
			T^{*} & B \\ 
			\hline
			C & D
		\end{array}\right)\in L(H\oplus\mathscr{E}_{*}, H^{d}\oplus\mathscr{E})
	\end{displaymath}
	satisfying the conditions (K1)-(K4) such that
	\begin{displaymath}
	W(z) = D + CF(ZT^{*})ZB \quad (z \in \mathbb{B}_{d}).
	\end{displaymath}
\end{thm}\label{thmkinnerunique}
\begin{proof}
	Since \(W\) is \(K\)-inner, the space
	\begin{displaymath}
		\mathscr{W}=W\mathscr{E}_{*}\subset H_{K}(\mathscr{E})
	\end{displaymath}
	is a generating wandering subspace for \(M_{z}\in L(H_{K}(\mathscr{E}))^{d}\) restricted to
	\begin{displaymath}
		\mathscr{S}=\bigvee_{\alpha\in\mathbb{N}^{d}}M_{z}^{\alpha}\mathscr{W}\subset H_{K}(\mathscr{E}).
	\end{displaymath}
	The compression \(T=P_{H}M_{z}\vert_{H}\) of \(M_{z} \in L(H_{K}(\mathscr{E}))^{d}\) 
	to the \(M_{z}^{*}\)-invariant subspace \(H=H_{K}(\mathscr{E})\ominus \mathscr{S}\) is easily seen to be 
	a pure \(K\)-contraction (\cite[Proposition 2.12 and Lemma 2.21]{Sch18}). 
	Let \(\mathscr{R}\subset H_{K}(\mathscr{E})\) be the smallest reducing subspace 
	for \(M_{z}\in L(H_{K}(\mathscr{E}))^{d}\) that contains \(H\). By Lemma \ref{lemmzreducing}
	\begin{displaymath}
	\mathscr{R}=\bigvee_{\alpha\in\mathbb{N}^{d}}z^{\alpha}(\mathscr{R}\cap\mathscr{E})=H_{K}(\mathscr R  \cap\mathscr{E}).
	\end{displaymath}
	Thus the inclusion map \(i\colon H\to H_{K}(\mathscr{R}\cap\mathscr{E})\) is a minimal 
	\(K\)-dilation for \(T\). Let \(j\colon H\to H_{K}(\mathscr{D}) \) be the \(K\)-dilation 
	of the pure \(K\)-contraction \(T\in L(H)^{d}\) defined in Theorem \ref{thminter}. 
	Since also \(j\) is a minimal \(K\)-dilation for \(T\) (Corollary \ref{corminimal}), 
	by Corollary \ref{corkdilation} there is a unitary operator 
	\(U\colon\mathscr{D}\to\mathscr{R}\cap\mathscr{E}\) such that
	\begin{displaymath}
		i = (1_{H_{K}}\otimes U)j.
	\end{displaymath}
	Define \(\hat{\mathscr{E}}=\mathscr{E}\ominus(\mathscr{R}\cap\mathscr{E})\). By construction
	\begin{displaymath}
		H_{K}(\hat{\mathscr{E}})=H_{K}(\mathscr{E})\ominus H_{K}(\mathscr{R}\cap\mathscr{E}) = 
		H_{K}(\mathscr{E})\ominus \mathscr R \subset \mathscr{S}
	\end{displaymath}
	is the largest reducing subspace for \(M_{z}\in L(H_{K}(\mathscr{E}))^{d}\) contained in \(\mathscr{S}\). In particular, the space \(\mathscr{S}\) admits the orthogonal decomposition
	\begin{displaymath}
		\mathscr{S}=H_{K}(\hat{\mathscr{E}})\oplus(\mathscr{S}\cap H_{K}(\hat{\mathscr{E}})^{\perp})=H_{K}(\hat{\mathscr{E}})\oplus(H_{K}(\mathscr{R}\cap\mathscr{E})\ominus \mathscr{S}^{\perp}).
	\end{displaymath}
	We complete the proof by comparing the given \(K\)-inner function \(W\colon\mathbb{B}_{d}\to L(\mathscr{E}_{*},\mathscr{E})\) 
	with the \(K\)-inner function \(W_{T}\colon\mathbb{B}_{d}\to L(\tilde{\mathscr{D}},\mathscr{D})\) associated with the 
	pure \(K\)-contraction \(T\in L(H)^{d}\). For this purpose, let us define the \(M_{z}\)-invariant subspace
	\begin{displaymath}
		M = H_{K}(\mathscr{D})\ominus\operatorname{Im}j
	\end{displaymath}
	and its wandering subspace
	\begin{displaymath}
		W(M) = M\ominus\left(\sum_{i=1}^{d}z_{i}M\right)
	\end{displaymath}
	as in Section \ref{sectionwanderingsubspaces}. Using the identity \(i=(1_{H_K}\otimes U)j\) one obtains that
	\begin{displaymath}
		1_{H_{K}}\otimes U\colon M\to H_{K}(\mathscr{R}\cap\mathscr{E})\ominus \mathscr{S}^{\perp} = 
		H_K(\mathscr{R}\cap\mathscr{E})\cap\mathscr{S}
	\end{displaymath}
	defines a unitary operator that intertwines the restrictions of \(M_{z}\) to both sides componentwise. Consequently we obtain the orthogonal decomposition
	\begin{align*}
		\mathscr{W}&=W_{M_{z}}(\mathscr{S})
		=W_{M_{z}}(H_{K}(\hat{\mathscr{E}}))\oplus W_{M_{z}}(H_{K}(\mathscr{R}\cap\mathscr{E})\cap\mathscr{S})\\
		&=\hat{\mathscr{E}}\oplus(1_{H_{K}}\otimes U)W(M).
	\end{align*}
	Let \(W_{T}\colon\mathbb{B}_{d}\to L(\tilde{\mathscr{D}},\mathscr{D})\) be the \(K\)-inner function, associated with the pure \(K\)-contraction \(T\in L(H)^{d}\). Then there is a matrix operator
	\begin{displaymath}
		\left(
		\begin{array}{c|c}
			T^{*}&B\\
			\hline
			C&D
			\end{array}
		\right)\in L(H\oplus\tilde{\mathscr{D}},H^{d}\oplus\mathscr{D})
	\end{displaymath}
 	such that
	\begin{displaymath}
		W_{T}(z)=D+CF(ZT^*)ZB \quad \; (z \in \mathbb{B}_{d})
	\end{displaymath} 
	and \(W(M)=\lbrace W_{T}x;\thickspace x\in\tilde{\mathscr{D}}\rbrace\) (see the beginning of Section \ref{sectionkinnerfunctions} and Theorem \ref{thmwandinner}). Let us denote by
	\begin{displaymath}
		P_{1}\colon\mathscr{W}\to\hat{\mathscr{E}}\text{ and } P_{2}\colon\mathscr{W}\to(1_{H_{K}}\otimes U)W(M)
	\end{displaymath}
	 the orthogonal projections. The \(K\)-inner functions \(W\colon\mathbb{B}_{d}\to L(\mathscr{E}_{*},\mathscr{E})\) and \(W_{T}\colon\mathbb{B}_{d}\to L(\tilde{\mathscr{D}},\mathscr{D})\) induce unitary operators 
	\begin{displaymath}
		\mathscr{E}_{*}\to\mathscr{W},\thickspace x\mapsto Wx
	\end{displaymath}
	and 
	\begin{displaymath}
		\tilde{\mathscr{D}}\to W(M)\thickspace x\mapsto W_{T}x.
	\end{displaymath}
	We define surjective bounded linear operators by
	\begin{displaymath}
	U_{1}\colon\mathscr{E}_{*}\to\hat{\mathscr{E}},\thickspace U_{1}x=P_{1}Wx
	\end{displaymath}
	and
	\begin{displaymath}
	U_{2}\colon\mathscr{E}_{*}\to\tilde{\mathscr{D}},\thickspace U_{2}x=\tilde{x}\text{ if }(1_{H_{K}}\otimes U)W_{T}x=P_{2}Wx.
	\end{displaymath}
	By construction the column operator
	\begin{displaymath}
		(U_{1},U_{2})\colon\mathscr{E}_{*}\to\hat{\mathscr{E}}\oplus\tilde{\mathscr{D}}
	\end{displaymath}
	defines an isometry such that
	\begin{displaymath}
		W(z)x=U_{1}x +UW_{T}(z)U_{2}x=(U_{1}+UDU_{2})x+(UC)F(ZT^*)Z(BU_{2})x
	\end{displaymath}
	holds for \(z\in\mathbb{B}_{d}\) and \(x\in\mathscr{E}_{*}\). To complete the proof we show that the operators 
	\begin{align*}
		&T\in L(H^{d},H),\thickspace\tilde{B}=BU_{2}\in L(\mathscr{E}_{*},H^{d}),\thickspace\tilde{C}=UC\in L(H,\mathscr{E})\\
		&\text{and }
		\tilde{D}=(U_{1}+UDU_{2})\in L(\mathscr{E}_{*},\mathscr{E})
	\end{align*}
	satisfy the conditions (K1)-(K4). To see this note that
	\begin{displaymath}
		\tilde{C}^{*}\tilde{C}=C^{*}U^{*}UC=C^{*}C=\frac{1}{K}(T)
	\end{displaymath}
	and
	\begin{align*}
		\tilde{D}^{*}\tilde{C}&=U_{2}^{*}D^{*}U^{*}UC
		=U_{2}^{*}D^{*}C\\
		&=-U_{2}^{*}B^{*}\left(\oplus\Delta_{T}\right)T^{*}
		=-\tilde{B}^{*}\left(\oplus\Delta_{T}\right)T^{*}.
	\end{align*}
	To verify condition (K3) note that \(\tilde{\mathscr{D}}\) acts as the column operator
	\begin{displaymath}
		\tilde{D}=(U_{1},UDU_{2})\colon\mathscr{E}_{*}\to\mathscr{E}=\hat{\mathscr{E}}\oplus(R\cap\mathscr{E}).
	\end{displaymath}
	Thus we obtain that
	\begin{align*}
		\tilde{D}^{*}\tilde{D}&=U_{1}U_{1}+U_{2}^{*}D^{*}U^{*}DU_{2}\\
		&=U_{1}^{*}U_{1}+U_{2}^{*}U_{2}-U_{2}^{*}B^{*}\left(\oplus\Delta_{T}\right)BU_{2}\\
		&=1_{\mathscr{E}_{*}}-\tilde{B}^{*}\left(\oplus\Delta_{T}\right)\tilde{B}.
	\end{align*}
	Since \(j_{\tilde{C}}=Uj_{C}\), it follows that
	\begin{displaymath}
		\left(\oplus j_{\tilde{C}}\right)\tilde{B}x=(\oplus U)(\oplus j_{C})B(U_{2}x)\in M_{z}^{*}H_{K}(\mathscr{E})
	\end{displaymath}
	holds for all \(x\in\mathscr{E}_{*}\). Thus the \(K\)-inner function \(W\colon\mathbb{B}_{d}\to L(\mathscr{E}_{*},\mathscr{E})\) 
	admits a matrix representation of the claimed form.
\end{proof}

\section{Minimal \(K\)-dilations}\label{sectionminimalkdilations}

Let \(\mathscr{A}\) be a unital subalgebra of a unital \(C^{*}\)-Algebra \(\mathscr{B}\). A completely positive unital map \(\varphi\colon\mathscr{B}\to L(H)\) is called an \(\mathscr{A}\)-morphism if
\(\varphi(1_{B})=1_{H}\) and \(\varphi(ax)=\varphi(a)\varphi(x)\) for  \(a\in\mathscr{A}\) and \(x\in\mathscr{B}\). Under the condition that \(\mathscr{B}\) is the norm-closed linear span
\begin{displaymath}
	\mathscr B = \overline{\operatorname{span}}^{\Vert\cdot\Vert}\lbrace\mathscr{A}\mathscr{A}^{*}\rbrace
\end{displaymath} 
Arveson proved in \cite[Lemmma 8.6]{Arv98} that every unitary operator that intertwines two \(\mathscr{A}\)-morphisms \(\varphi_{i}\colon\mathscr{B}\to L(H_{i})\) \((i=1,2)\) pointwise on \(\mathscr{A}\) extends to a unitary operator that intertwines the minimal Stinespring representations of \(\varphi_{1}\) and \(\varphi_{2}\).\\
Straightforward modifications of the arguments given in \cite{Arv98} show that Arveson's result remains true if \(\mathscr{B}\) is a von Neumann algebra  which is the \(w^{*}\)- closed linear span
\begin{displaymath}
\mathscr B = \overline{\operatorname{span}}^{w^{*}}\lbrace\mathscr{A}\mathscr{A}^{*}\rbrace
\end{displaymath}
and if the \(\mathscr{A}\)-morphisms \(\varphi_{i}\colon\mathscr{B}\to L(H_{i})\) \((i=1,2)\) are supposed to be \(w^{*}\)-continuous
 
\begin{thm}\label{thmarveson}
	Let \(\mathscr{B}\) be a von Neumann algebra and let \(\mathscr{A}\subset\mathscr{B}\) be a unital subalgebra such that
	\begin{displaymath}
		\mathscr{B}=\overline{\operatorname{span}}^{w^{*}}\lbrace\mathscr{A}\mathscr{A}^{*}\rbrace.
	\end{displaymath}
	 For \(i=1,2\), let \(\varphi_{i}\colon\mathscr{B}\to L(H_{i})\) be a \(w^{*}\)-continuous 
	\(\mathscr{A}\)-morphism and let \((\pi_{i},V_{i},H_{\pi_{i}})\) be the minimal Stinespring 
	representations for \(\varphi_{i}\). For every unitary operator \(U\colon H_{1}\to H_{2}\) with 
	\begin{displaymath}
	U\varphi_{1}(a)=\varphi_{2}(a)U\quad(a\in\mathscr{A}),
	\end{displaymath}
	there is a unique unitary operator \(W\colon H_{\pi_{1}}\to H_{\pi_{2}}\) with \(WV_1=V_2U\) and 
	\(W\pi_{1}(x)=\pi_{2}(x)W\) for all \(x\in\mathscr{B}\). 
	\end{thm}
Since this version of Arveson's result follows in exactly the same way as the original one (\cite[Lemmma 8.6]{Arv98}), 
we leave the details to the reader.
\\
As an application of Theorem \ref{thmarveson} we show that, under suitable conditions on the kernel 
\(K\colon\mathbb{B}_{d}\times\mathbb{B}_{d}\to\mathbb{C}\), minimal \(K\)-dilations are uniquely 
determined. Recall that a commuting tuple \(T\in L(H)^{d}\) on a Hilbert space \(H\) is called 
essentially normal if \(T_{i}T_{i}^{*}-T_{i}^{*}T_{i}\) is compact for \(i=1,\ldots,d\). 
If \(T\in L(H)^{d}\) is essentially normal, then by the Fuglede-Putnam theorem also all cross 
commutators \(T_{i}T_{j}^{*}-T_{j}^{*}T_{i}\) \((i,j=1,\ldots,d)\) are compact. For our multiplication 
tuple \(M_{z}\in L(H_K)^{d}\), essential normality is equivalent to the condition that (\cite[Corollary 4.4]{Guo04})
\begin{displaymath}
\lim_{n\to\infty}\left(\frac{a_{n}}{a_{n+1}}-\frac{a_{n-1}}{a_{n}}\right)=0.
\end{displaymath}

\begin{thm}\label{thmvonneumannmz}
Suppose that \(M_{z}\in L(H_K)^{d}\) is an essentially normal \(K\)-contra\-ction. 
Then the von Neumann algebra generated by \(M_{z_{1}},\ldots,M_{z_{d}}\) is given by
\begin{displaymath}
W^{*}(M_{z})=\overline{\operatorname{span}}^{w^{*}}\lbrace M_{z}^{\alpha}M_{z}^{*\beta};\thickspace\alpha,\beta\in\mathbb{N}^{d}\rbrace.
\end{displaymath}
\end{thm}
\begin{proof}
Define \(\mathscr L=\overline{\operatorname{span}}^{w^{*}}\lbrace M_z^{\alpha}M_z^{*\beta}; \; \alpha,\beta\in\mathbb{N}^{d}\rbrace\) . Obviously \(\mathscr{L}\subset W^{*}(M_{z})\). Since \(M_{z}\) is supposed to be a \(K\)-contraction,
	\begin{displaymath}
		P_{\mathbb{C}}=\tau_{\operatorname{SOT}}-\sum_{n=0}^{\infty}c_{n}\sigma_{M_{z}}^{n}(1_{H_K})\in\mathscr{L}.
	\end{displaymath}
	For \(\alpha,\beta\in\mathbb{N}^{d}\) and \(w\in\mathbb{B}_{d}\), we obtain
	\begin{displaymath}
		M_{z}^{\alpha}P_{\mathbb{C}}M_{z}^{*\beta}(K(\cdot,w))=\overline{w}^{\beta}z^{\alpha}=z^{\alpha}\otimes z^{\beta}(K(\cdot,w)).
	\end{displaymath}
	Since the multiplication on \(L(H_K)\) is separately \(w^{*}\)-continuous, it follows that $\mathscr L$ contains all
	compact operators
	\begin{displaymath}
		K(H_K)=\overline{\operatorname{span}}^{\Vert\cdot\Vert}\lbrace z^{\alpha}\otimes z^{\beta};\thickspace\alpha,\beta\in\mathbb{N}^{d}\rbrace\subset\mathscr{L}.
	\end{displaymath}
	But then the hypothesis that \(M_{z}\) is essentially normal implies that \(\mathscr{L}\subset L(H_K)\) is a subalgebra. Since the involution on \(L(H_K)\) is \(w^{*}\)-continuous, the algebra \(\mathscr{L}\subset L(H_K)\) is a von Neumann algebra and hence \(\mathscr{L}=W^{*}(M_{z})\).
\end{proof}

The tuple \(M_{z}\in L(H_K)^{d}\) is known to be a \(K\)-contraction if there is a natural number 
\(p\in\mathbb{N}\) such that \(c_{n}\geq 0\) for all \(n\geq p\) or \(c_{n}\leq 0\) for all 
\(n\geq p\) (\cite[Lemma 2.2]{Chen} or \cite[Proposition 2.10]{Sch18}). The latter condition holds, 
for instance, if \(H_K\) is a complete Nevanlinna-Pick space such as the Drury-Arveson or Dirichlet space
on the unit ball or if \(K\) is a kernel of the form
\begin{align*}
	K_{\nu}\colon\mathbb{B}_{d}\times\mathbb{B}_{d}\to\mathbb{C},K_{\nu}(z,w)=\frac{1}{(1-\langle z,w\rangle)^{\nu}}
\end{align*} 
with a positive real number \(\nu>0\).
\\\\
Let \(T\in L(H)^{d}\) be a commuting tuple and let \(j\colon H\to H_{K}(\mathscr{E})\) be a \(K\)-dilation of \(T\). We denote by \(\mathscr{B}=W^{*}(M_{z})\subset L(H_K)\) the von Neumann algebra generated by \(M_{z}\) and set \(\mathscr{A}=\lbrace p(M_{z});\thickspace p\in\mathbb{C}[z]\rbrace\). The unital \(C^{*}\)-homomorphism
\begin{displaymath}
	\pi\colon\mathscr{B}\to L(H_{K}(\mathscr{E})),\thickspace X\mapsto X\otimes 1_{\mathscr{E}}
\end{displaymath}
together with the isometry \(j\colon H\to H_{K}(\mathscr{E})\) is a Stinespring representation for the completely positive map
\begin{displaymath}
\varphi\colon\mathscr{B}\to L(H_{K}(\mathscr{E})),\thickspace \varphi(X)=j^{*}(X\otimes 1_{\mathscr{E}})j.
\end{displaymath}
The map \(\varphi\) is an \(\mathscr{A}\)-morphism, since
\begin{align*}
	\varphi(p(M_{z})X)&=j^{*}(p(M_{z}\otimes 1_{\mathscr{E}})X\otimes 1_{\mathscr{E}})j = 
	j^{*}p(M_{z}\otimes 1_{\mathscr{E}})(jj^{*})(X\otimes 1_{\mathscr{E}})j\\&=\varphi(p(M_{z}))\varphi(X)
\end{align*}
for all \(p\in\mathbb{C}[z]\) and \(X \in\mathscr{B}\). Standard duality theory for Banach space operators 
shows that \(\pi\) is \(w^{*}\)-continuous. Indeed, as an application of Krein-Smulian's theorem 
(Theorem IV. 6.4 in \cite{Sae66}) one only has to check that 
\(\tau_{w^{*}}-\lim_{\alpha}(X_{\alpha}\otimes 1_{\mathscr{E}}) = X\otimes 1_{\mathscr{E}}\) 
for each norm-bounded net \((X_{\alpha})\) in \(\mathscr{B}\) with \(\tau_{w^{*}}-\lim_{\alpha}X_{\alpha}=X\). 
To complete the argument it suffices to recall that on norm-bounded sets the \(w^{*}\)-topology 
and the weak operator topology coincide. Thus we have shown that \(\varphi\) is a \(w^{*}\)-continuous 
\(\mathscr{A}\)-morphism with Stinespring representation \(\pi\). By definition the \(K\)-dilation 
\(j\colon H\to H_{K}(\mathscr{E})\) is minimal if and only if
\begin{displaymath}
	\bigvee_{X\in W^{*}(M_{z})}\pi(X)(jH)=H_{K}(\mathscr{E}),
\end{displaymath} 
hence if and only if \(\pi\) as a Stinespring representation of \(\varphi\) is minimal.

\begin{cor}\label{corkdilation}
	Suppose that \(M_{z}\in L(H_K)^{d}\) is an essentially normal \(K\)-contraction. If \(j_{i}\colon H\to H_{K}(\mathscr{E}_{i})\) \((i=1,2)\) are two minimal \(K\)-dilations of a commuting tuple \(T\in L(H)^{d}\), then there is a unitary operator \(U\in L(\mathscr{E}_{1},\mathscr{E}_{2})\) such that \(j_{2}=(1_{H_{K}}\otimes U)j_{1}\)
\end{cor}
\begin{proof}
	As before we denote by \(\mathscr{B}=W^{*}(M_{z})\subset L(H_K)\) the von Neumann algebra generated by \(M_{z_{1}},\ldots,M_{z_{d}}\in L(H_K)\) and define \(\mathscr{A}=\lbrace p(M_{z});\thickspace p\in\mathbb{C}[z]\rbrace\). The remarks preceding the corollary show that the maps
	\begin{displaymath}
		\varphi_{i}\colon\mathscr{B}\to L(H),\thickspace\varphi_{i}(X)=j_{i}^{*}(X\otimes 1_{\mathscr{E}_{i}})j_{i}\quad(i=1,2)
	\end{displaymath}
	are \(w^{*}\)-continuous \(\mathscr{A}\)-morphisms with minimal Stinespring representations
	\begin{displaymath}
		\pi_{i}\colon\mathscr{B}\to L(H_{K}(\mathscr{E}_{i})),\thickspace\pi_{i}(X)=X\otimes 1_{\mathscr{E}_{i}}\quad(i=1,2).
	\end{displaymath}
	Since
	\[
	\varphi_i(p(M_z)) = j^* p(M_z \otimes 1_{\mathscr E}) j = p(T)
	\]
	for all $p \in \mathbb C[z]$ and $i = 1,2$,
	Theorem \ref{thmarveson} implies that there is a unitary operator \(W\colon H_{K}(\mathscr{E}_{1})\to H_K(\mathscr{E}_{2})\) 
	with \(Wj_{1}=j_{2}\) and \(W(X\otimes 1_{\mathscr{E}_{1}})=(X\otimes 1_{\mathscr{E}_{2}})W\)
	for all \(X\in\mathscr{B}\). In particular, the unitary operator \(W\) satisfies the intertwining relations 
	\begin{displaymath}
		W(M_{z_{i}}\otimes 1_{\mathscr{E}_{1}})=(M_{z_{i}}\otimes 1_{\mathscr{E}_{2}})W\quad(i=1,\ldots,d)
	\end{displaymath}
	A standard characterization of multipliers on reproducing kernel Hilbert spaces 
	(\cite[Theorem 2.1]{Bar11}) shows that there exist operator-valued functions \(A\colon\mathbb{B}_{d}\to L(\mathscr{E}_{1},\mathscr{E}_{2})\) and \(B\colon\mathbb{B}_{d}\to L(\mathscr{E}_{2},\mathscr{E}_{1})\) such that \(Wf=Af\) and \(W^{*}g=Bg\) for \(f\in H_{K}(\mathscr{E}_{1})\) and \(g\in H_{K}(\mathscr{E}_{2})\) (see also \cite[Proposition 4.5]{Sch18}). It follows that \(A(z)B(z)=1_{\mathscr{E}_{2}}\) and \(B(z)A(z)=1_{\mathscr{E}_{1}}\) for \(z\in\mathbb{B}_{d}\). Since
	\begin{displaymath}
		K(z,w)x=(WW^{*}K(\cdot,w)x)(z)=A(z)K(z,w)A(w)^{*}x
	\end{displaymath}
	for \(z,w\in\mathbb{B}_{d}\) and \(x\in\mathscr{E}_{2}\), we find that \(A(z)A(w)^{*}=1_{\mathscr{E}_{2}}\) for \(z,w\in\mathbb{B}_{d}\). But then the constant value \(A(z)\equiv U\in L(\mathscr{E}_{1},\mathscr{E}_{2})\) is a unitary operator with \(W=1_{H_K}\otimes U\).
\end{proof}

We conclude this section by showing that the canonical \(K\)-dilation of a \(K\)-contraction  \(T\in L(H)^{d}\) defined in Theorem \ref{thminter} is minimal. To prepare this result we first identify the \(M_{z}\)-reducing subspaces of \(H_{K}(\mathscr{E})\).

\begin{lem}\label{lemmzreducing}
	Let \(M\subset H_{K}(\mathscr{E})\) be a closed linear subspace. If \(M\) is reducing for \(M_{z}\in L(H_{K}(\mathscr{E}))^{d}\), then \(P_{\mathscr{E}}M\subset M\) and
	\begin{displaymath}
		M=\bigvee_{\alpha\in \mathbb{N}^{d}}z^{\alpha}(M\cap\mathscr{E})=H_{K}(M\cap\mathscr{E}).
	\end{displaymath}
\end{lem}
\begin{proof}
The hypothesis implies that \(M\) is reducing for the von Neumann algebra \(W^{*}(M_{z})\subset L(H_{K}(\mathscr{E}))\) generated by \(M_{z_{1}},\ldots M_{z_{d}}\in L(H_{K}(\mathscr{E}))\). Standard results on von Neumann algebras (Corollary 17.6 and Proposition 24.1 in \cite{Zhu93})
show that
\begin{displaymath}
	P_{\mathscr{E}}=P_{\bigcap\operatorname{Ker}M_{z_{i}}^{*}}\in W^{*}(M_{z}).
\end{displaymath}
Hence \(P_{\mathscr{E}}M\subset M\). Let \(f=\sum_{\alpha\in\mathbb{N}^{d}}f_{\alpha}z^{\alpha}\in H_{K}(\mathscr{E})\) be arbitrary. An elementary calculation yields that
\begin{displaymath}
	P_{\mathscr{E}}(M_{z}^{*\beta}f)\in(\mathbb{C}\setminus\lbrace 0\rbrace)f_{\beta}\quad(\beta\in\mathbb{N}^{d}).
\end{displaymath}
Hence, if \(f\in M\), then \(f_{\beta}\in M\cap\mathscr{E}\) for all \(\beta\in\mathbb{N}^{d}\) and the observation that
\begin{displaymath}
	f=\sum_{\alpha\in\mathbb{N}^{d}}f_{\alpha}z^{\alpha}\in\bigvee_{\alpha\in \mathbb{N}^{d}}z^{\alpha}(M\cap\mathscr{E})=H_{K}(M\cap\mathscr{E})
\end{displaymath}
completes the proof.
\end{proof}
 
\begin{cor}\label{corminimal}
Let \(T\in L(H)^{d}\) be a pure \(K\)-contraction. Then the \(K\)-dilation
\begin{displaymath}
	j\colon H\to H_{K}(\mathscr{D}),\thickspace j(x)=\sum_{\alpha\in\mathbb{N}^{d}}a_{\vert\alpha\vert}\gamma_{\alpha}(CT^{*\alpha}x)z^{\alpha}
\end{displaymath}
defined in Theorem \ref{thminter} is minimal. 
\end{cor}
\begin{proof}
Let \(\operatorname{Im}j\subset M\) be a reducing subspace for \(M_{z}\in L(H_{K}(\mathscr{D}))^{d}\). We know from Lemma \ref{lemmzreducing} that
\begin{displaymath}
 	M=\bigvee_{\alpha\in \mathbb{N}^{d}}z^{\alpha}(M\cap\mathscr{D})
\end{displaymath}
 and that
\begin{displaymath}
	CH=P_{\mathscr{D}}(\operatorname{Im}j)\subset P_{\mathscr{D}}(M)\subset M\cap\mathscr{D}.
\end{displaymath}
It follows that \(\mathscr{D}=\overline{CH}=M\cap\mathscr{D}\) and that \(M=\bigvee_{\alpha\in \mathbb{N}^{d}}z^{\alpha}\mathscr{D}=H_{K}(\mathscr{D})\).
\end{proof}


\end{document}